\documentclass{amsart}
\usepackage[all]{xy}
\usepackage[pagebackref=false,hyperindex=true,colorlinks=false, pdfauthor={Olivier Haution}, pdftitle={Lifting of coefficients for {C}how motives of quadrics}, pdfkeywords={quadratric forms, {C}how groups, motivic decompositions  }]{hyperref}

\DeclareMathOperator{\Hom}{Hom}
\DeclareMathOperator{\CH}{CH}
\DeclareMathOperator{\H1}{H^1}
\DeclareMathOperator{\I}{I}
\DeclareMathOperator{\Spec}{Spec}
\DeclareMathOperator{\Gal}{Gal}
\DeclareMathOperator{\Aut}{Aut}
\DeclareMathOperator{\im}{im}
\DeclareMathOperator{\End}{End}

\DeclareMathOperator{\disc}{disc}
\DeclareMathOperator{\GL}{GL}
\DeclareMathOperator{\SL}{SL}

\newcommand{\MZ}{\mathcal{CM}(\mathcal{Q}_F,\mathbb{Z})}
\newcommand{\MZZ}{\mathcal{CM}(\mathcal{Q}_F,{\mathbb{Z}/2})}
\newcommand{\MZZn}{\mathcal{CM}(\mathcal{Q}_F,{\mathbb{Z}/2^n})}
\newcommand{\ML}{\mathcal{CM}(\mathcal{Q}_F,\Lambda)}
\newcommand{\MLL}{\mathcal{CM}(\mathcal{Q}_L,\Lambda)}

\newcommand{\CZZ}{\mathcal{C}(\mathcal{Q}_F,{\mathbb{Z}/2})}
\newcommand{\CZZn}{\mathcal{C}(\mathcal{Q}_F,{\mathbb{Z}/2^n})}
\newcommand{\CL}{\mathcal{C}(\mathcal{Q}_F,\Lambda)}

\newcommand{\Q}{\mathcal {Q}}

\newcommand{\Gf}{G(\varphi)}
\newcommand{\rL}{{(\rho_L)}_*}
\newcommand{\sL}{{(\sigma_L)}_*}
\newcommand{\sep}{\bar{F}}
\newcommand{\Z}{\mathbb{Z}}
\newcommand{\AutL}{\Aut (L/F)}
\newcommand{\GalL}{\Gal (L/F)}

\newtheorem{theorem}[equation]{Theorem}
\newtheorem{proposition}[equation]{Proposition}
\newtheorem{corollary}[equation]{Corollary}
\newtheorem{lemma}[equation]{Lemma}

\title{Lifting of coefficients for Chow motives of quadrics}
\author{Olivier Haution}
\address{Institut de Math\'ematiques de Jussieu, Universit\'e Pierre et Marie Curie - Paris 6, 4 place Jussieu, F-75252 Paris CEDEX 05, FRANCE.}
\email{olivier.haution@gmail.com}
\date{September 17, 2009}

\begin{document}

\begin{abstract}
We prove that the natural functor from the category of Chow motives of smooth projective quadrics with integral coefficients to the category with coefficients modulo $2$ induces a bijection on the isomorphism classes of objects. 
\end{abstract}
\maketitle

\tableofcontents

\section{Introduction}
Alexander Vishik has given a description of the Chow motives of quadrics with integral coefficients in \cite{V}. It uses much subtler methods than the ones used to give a similar description with coefficients in $\Z/2$, found for example in \cite{EKM}, but the description obtained is the same (\cite{EKM}, Theorems $93.1$ and $94.1$). The result presented here allows to recover Vishik's results from the modulo $2$ description. 

In order to state the main result, we first define the categories involved.  Let $\Lambda$ be a commutative ring. We write $\Q_F$ for the class of smooth projective quadrics over a field $F$. We consider the additive category $\CL$, where objects are (coproducts of) quadrics in $\Q_F$ and if $X,Y$ are two such quadrics, $\Hom(X,Y)$ is the group of correspondences of degree $0$, namely $\CH_{\dim X}(X \times Y,\Lambda)$. We write $\ML$ for the idempotent completion of $\CL$. This is the category of graded Chow motives of smooth projective quadrics with coefficients in $\Lambda$. If $(X,\rho), (Y,\sigma)$ are two such motives then we have :
\[\Hom((X,\rho),(Y,\sigma))=\sigma \circ \CH_{\dim X}(X \times Y,\Lambda) \circ \rho.\]
We will prove the following :
\begin{theorem}
\label{coeff}
The functor $\MZ \to \MZZ$ induces a bijection on the isomorphism classes of objects. 
\end{theorem}

The proof mostly relies on the low rank of the homogeneous components of the Chow groups of quadrics when passing to a splitting field. These components are almost always indecomposable if we take into account the Galois action. The only exception is the component of rank $2$ when the discriminant is trivial but in this case the Galois action on the Chow group is trivial which allows the proof to go through. 

It seems that Theorem \ref{coeff} may be deduced from \cite{V} (see Theorem E.11.2 p.254 in \cite{Ka}). Here we try to give a more direct and self-contained proof.\\

This work is part of my Ph.D. thesis at the University of Paris 6 under the direction of Nikita Karpenko. I am very grateful to him for his useful suggestions.

\section{Chow groups of quadrics}
We first recall some facts and fix the notations that we will use.\\

If $L/F$ is a field extension, and $S$ a scheme over $F$, we write $S_L$ for the scheme $S \times_{\Spec(F)} \Spec(L)$. Similarly, for an $F$-vector space $U$, we write $U_L$ for $U \otimes_F L$, and for a cycle $x \in \CH(S)$, the element $x_L \in \CH(S_L)$ is the pull-back of $x$ along the flat morphism $S_L \to S$.

We say that a cycle in $\CH(S_L)$ is \emph{$F$-rational} (or simply \emph{rational} when no confusion seems possible) if it can be written as $x_L$ for some cycle $x \in \CH(S)$, \emph{i.e.} if it belongs to the image of the pull-back homomorphism $\CH(S) \to \CH(S_L)$.\\

Let $F$ be a field and $\varphi$ be a non-degenerate quadratic form on a $F$-vector space $V$ of dimension $D+2$.  The associated projective quadric $X$ is smooth of dimension $D=2d$ or $2d+1$. Let $L/F$ be a splitting extension for $X$, \emph{i.e.} a field extension such that $V_L$ has a totally isotropic subspace of dimension $d+1$. We write $h^i,l_i$ for the usual basis of $\CH(X_L)$, where $0 \leq i \leq d$. The class $h$ is the pull-back of the hyperplane class of the projective space of $V_L$, the class $l_i$ is the class of the projectivisation of a totally isotropic subspace of $V_L$ of dimension $i+1$.

If $D$ is even, then $\CH_d(X_L)$ is freely generated by $h^d$ and $l_d$. In this case, there are exactly two classes of maximal totally isotropic spaces, $l_d$ and ${l_d}'$. They correspond to spaces exchanged by a reflection and verify the relation $l_d + {l_d}'=h^d$.

The group $\AutL$ acts on $\CH(X_L)$. It acts trivially on the $i$-th homogeneous component of $\CH(X_L)$, as long as $2i \neq D$. 

See \cite{EKM} for proofs of all these facts.\\

In the next proposition, $X$ is a smooth projective quadric of dimension $D=2d$ associated with a quadratic space $(V,\varphi)$ over a field $F$, $L/F$ is a splitting extension for $X$, and $\disc X$ is the discriminant algebra of $\varphi$. 

\begin{proposition}
\label{disc}
Under the natural $\AutL$-actions, we can identify the pair $\{l_d, {l_d}'\}$ and the connected components of $\Spec(\disc X \otimes L)$.
\end{proposition}
\begin{proof}
  We consider the scheme $\Gf$ of maximal totally isotropic subspaces of $V$, \emph{i.e} the grassmannian variety of isotropic $(d+1)$-dimensional subspaces of $V$. The scheme $\Gf_L$ has two connected components exchanged by any reflection of the quadratic space $(V,\varphi)$. There is a faithfully flat morphism $\Gf \to \Spec(\disc X)$ (see \cite{EKM}, \S85, p.357), hence the connected components of $\Gf_L$ are in correspondence with those of $\Spec(\disc X \otimes L)$, in a way respecting the natural $\AutL$-actions. 

Now two maximal totally isotropic subspaces lie in the same connected component of $\Gf_L$ if and only if the corresponding $d$-dimensional closed subvarieties of the quadric $X_L$ are rationally equivalent (see \cite{EKM}, \S86, p.358). Therefore the pair  $\{l_d, {l_d}'\}$ is $\AutL$-isomorphic to the pair of connected components of $\Gf_L$. The statement follows.
\end{proof}

\section{Lifting of coefficients}

We now give a useful characterization of rational cycles (Proposition~\ref{r}). The proof will rely on the following theorem (\cite{Rost}, Proposition~9):
\begin{theorem}[Rost's nilpotence for quadrics]
Let $X$ be a smooth projective quadric over a field $F$, and let $\alpha \in \End_{\ML} (X)$. If $\alpha_L \in \CH(X_L^2)$ vanishes for some field extension $L/F$, then $\alpha$ is nilpotent.
\end{theorem}

We will use the following classical corollaries :
\begin{corollary}[\cite{EKM} Corollary~92.5]
\label{proj}
Let $X$ be a smooth projective quadric over a field $F$ and $L/F$ a field extension. Let $\pi$ a projector in $\End_{\MLL}(X_L)$ that is the restriction of some element in  $\End_{\ML}(X)$. Then there exist a projector $\varphi$ in  $\End_{\ML}(X)$ such that $\varphi_L=\pi$.
\end{corollary}

\begin{corollary}[\cite{EKM} Corollary~92.7]
\label{iso}
Let $f\colon(X,\rho) \to (Y,\sigma)$ be a morphism in $\ML$. If $f_L$ is an isomorphism for some field extension $L/F$ then $f$ is an isomorphism. 
\end{corollary} 

\begin{proposition}
\label{n}
For any $n \geq 1$, the functor $\MZZn \to \MZZ$ is bijective on the isomorphism classes of objects.
\end{proposition}
\begin{proof}
We are clearly in the situation $(\star)$ of \cite{VY}, \S~2 (p.587) for the functor $\CZZn \to \CZZ$ . The statement then follows from Propositions~2.5 and 2.2 of  \cite{VY} (see also Corollary 2.7 of \cite{PSZ}).
\end{proof}

Any smooth projective quadric admits a (non-canonical) finite Galois splitting extension, of degree a power of $2$. This will be used together with the following proposition when we will need to prove that a given cycle with integral coefficients (and defined over some extension of the base field) is rational.

\begin{proposition}
\label{r}
Let $X,Y \in \Q_F$ and $L/F$ be a splitting Galois extension of degree $m$ for $X$ and $Y$. A correspondence in $\CH((X\times Y)_L,\Z)$ is rational if and only if it is invariant under the group $\GalL$ and its image in $\CH((X \times Y)_L,\Z/m)$ is rational.
\end{proposition}  
\begin{proof}
We write $Z$ for $X \times Y$. We first prove that if $x$ is a $\GalL$-invariant cycle in $\CH(Z_L)$ then $[L:F]\cdot x$ is rational.

Let $\tau \colon L \to \sep$ be a separable closure so that we have an $\sep$-isomorphism $L \otimes \sep \to \sep \times ... \times \sep$ given by $u\otimes 1 \mapsto (\tau \circ \gamma(u))_{\gamma \in \GalL}$.

We have a cartesian square :
\[\xymatrix{
Z_L \ar[d]& \ar[l] \ar[d] Z_{L\otimes \sep}\\
Z & \ar[l] Z_{\sep} \\
}
\]
It follows that we have a commutative diagram of pull-backs and push-forwards:
\[\xymatrix{
\CH(Z_L) \ar[d] \ar[r]& \ar[d] \CH(Z_{L\otimes \sep})\\
\CH(Z) \ar[r] & \CH(Z_{\sep}) \\
}
\]
The top map followed by the map on the right is:
\[x \mapsto \sum_{\gamma \in \GalL} t^*(^\gamma x)\]
where $t:Z_{\sep} \to Z_L$ is the map induced by $\tau$. Using the commutativity of the diagram and the injectivity of $t^*$, we see that the composite $\CH(Z_L) \to \CH(Z) \to \CH(Z_L)$ maps $x$ to $\sum {}^\gamma x$, where $\gamma$ runs in $\GalL$. The claim follows.

Now suppose that $u$ is a cycle in $\CH(Z_L,\Z)$ invariant under $\GalL$, and that its image in $\CH(Z_L,\Z/m)$ is rational. We can find a rational cycle $v$ in $\CH(Z_L,\Z)$ and a cycle $\delta$ in $\CH(Z_L,\Z)$ such that $m \delta=v-u$. Since $\CH(Z_L,\Z)$ is torsion-free, $\delta$ is invariant under $\GalL$ . The first claim ensures that $v-u$ is rational, hence $u$ is rational.
\end{proof}

Let us remark that if $X\in \Q_F$, $L/F$ is a splitting extension, and $2i < \dim X$ then $2 l_i=h^{\dim X-i} \in \CH(X_L)$ is always rational. It follows that $2\CH_i(X_L)$ consists of rational cycles when $2i \neq \dim X$.

\section{Surjectivity in main Theorem }
\begin{proposition}
The functor $\MZ \to \MZZ$ is surjective on the isomorphism classes of objects.
\end{proposition}
\begin{proof}
Let $(X,\pi) \in \MZZ$ and $L/F$ a finite splitting Galois extension for $X$ of degree $2^n$. By Proposition~\ref{n}, we can lift the isomorphism class of $(X,\pi)$ to the isomorphism class of some $(X,\tau) \in \MZZn$. \\

Assume that we have found a $\GalL$-invariant projector $\rho$ in $\CH_{\dim X}(X \times X)_L$ which gives modulo $2^n$ the projector $\tau_L$. By Proposition~\ref{r}, $\rho$ is a rational cycle, and Corollary~\ref{proj} provides a projector $p$ such that $\rho = p_L$. Write $\tilde p \in \CH(X \times X,\mathbb{Z}/2^n)$ for the image of $p$. Consider the morphism $(X,\tau) \to (X,\tilde p)$ given by $\tilde p \circ \tau$. Since $\tilde p_L=\rho \mod 2^n=\tau_L$, this morphism becomes an isomorphism (the identity) after extending scalars to $L$ hence is an isomorphism by Corollary~\ref{iso}. It follows that the isomorphism class of $(X,p) \in \MZ$ is a lifting of the class of $(X,\tau) \in \MZZn$.\\

We now build the projector $\rho$. For any commutative ring $\Lambda$, projectors in $\CH((X \times X)_L,\Lambda)$ are in bijective correspondence with ordered pairs of subgroups of $\CH(X_L,\Lambda)$ which form a direct sum decomposition. This bijection is compatible with the natural $\GalL$-actions. A projector is of degree $0$ if and only if the two summands in the associated decomposition are graded subgroups of $\CH(X_L,\Lambda)$. 

When $\dim X$ is odd or when $\disc X$ is a field, each homogeneous component of $\CH(X_L,\Lambda)$ is $\GalL$-indecomposable, hence $\GalL$-invariant projectors of degree $0$ of $\CH(X_L,\Lambda)$ are in one-to-one correspondence with the subsets of $\{0,\cdots,\dim X\}$. It follows that we can lift any $\GalL$-invariant projector of degree $0$ with coefficients in $\Z/2^n$ to an integral $\GalL$-invariant projector of degree $0$. 

When $\dim X$ is even and $\disc X$ is trivial, $\CH_i(X_L,\Lambda)$ is indecomposable if $2i \neq 2d_X=\dim X$. The group $\GalL$ acts trivially on $\CH_{d_X}(X_L,\Lambda)$. If the rank of the restriction of $(\tau_L)_*$ to $\CH_{d_X}(X_L,\mathbb{Z}/2^n)$ is $0$ or $2$, the projector $\tau_L$ clearly lifts to a $\GalL$-invariant projector in $\CH_{\dim X}(X \times X)_L$.

The last case is when the rank is $1$. We fix a decomposition of the group $\CH_{d_X}(X_L,\mathbb{Z})$ into the direct sum of rank $1$ summands. Any such decomposition of $\CH_{d_X}(X_L,\Lambda)$ is then given by some element of $\SL_2(\Lambda)$. The next lemma ensures that we can lift any element of $\SL_2(\mathbb{Z}/2^n)$ to $\SL_2(\mathbb{Z})$, thus that $\tau_L$ lifts to a projector with integral coefficients. It remains to notice that $\GalL$ acts trivially on $\CH_{\dim X}(X \times X)_L$ since $\disc X$ is trivial, to conclude the proof.
\end{proof}

\begin{lemma}[\cite{PSZ}, Lemma~2.14]
\label{sl}
For any positive integers $k$ and $p$, the reduction homomorphism $\SL_k(\Z) \to \SL_k(\Z/p)$ is surjective.
\end{lemma}
\begin{proof}
  Since $\Z/p$ is a semi-local commutative ring, it follows from Corollary~9.3, Chapter~V, p.267 of \cite{Ba}, applied with $A=\Z$ and $\underline{q}=p\Z$ , that any matrix in $\SL_k(\Z / p)$ is the image modulo $p$ of a product of elementary matrices with integral coefficients. Such a product in particular belongs to $\SL_k(\Z)$, as required.
\end{proof}

\section{Injectivity in main Theorem}

In order to prove injectivity in Theorem~\ref{coeff}, we may assume that we are given two motives $(X,\rho), (Y,\sigma)$ in $\MZ$ and an isomorphism between their images in $\MZZ$. We will build an isomorphism with integral coefficients between the two motives (which will not, in general, be an integral lifting of the original isomorphism with finite coefficients).

We fix a finite Galois splitting extension $L/F$ for $X$ and $Y$ of degree $2^n$. Using Proposition~\ref{n} we may assume that there exists an isomorphism $\alpha$ between $(X,\rho)$ and $(Y,\sigma)$ in $\MZZn$. By Proposition \ref{r} and Corollary \ref{iso}, it is enough to build an isomorphism $(X_L,\rho_L) \to (Y_L,\sigma_L)$ which reduces to a rational correspondence modulo $2^n$ and which is equivariant under $\GalL$.\\

Let $d_X$ be such that $\dim X=2d_X$ or $2 d_X +1$ and $d_Y$ defined similarly for $Y$. Let $r(X,\rho)$ be the rank of $\CH_{d_X}(X_L) \cap \im{\rL}$ if $\dim X$ is even and $r(X,\rho)=0$ if $\dim X$ is odd. We define $r(Y,\sigma)$ in a similar fashion. We will distinguish cases using these integers.

A basis of $\CH(X_L)\cap \im \rL$ gives an isomorphism of $(X_L,\rho_L)$ with twists of Tate motives, thus choosing bases for the groups $\CH(X_L)\cap \im \rL$ and $\CH(Y_L)\cap \im \sL$, we can see morphisms between the two motives as matrices.

We fix a basis $(e_i)$ of $\CH(X_L) \cap \im \rL$ as follows : we choose $e_i \in \CH_i(X_L)$ among the cycles $h^{\dim X - i},l_i$ for $2i \neq \dim X$. We are done when $r(X,\rho)=0$.

If $r(X,\rho)=2$ we complete the basis with $e_{d_X}=l_{d_X}, e_{d_X}'=l_{d_X}' \in \CH_{d_X}(X_L)$.

If $r(X,\rho)=1$ we choose a generator $e_{d_X}$ of $\CH_{d_X}(X_L) \cap \im \rL$ to complete the basis.

We choose a basis $(f_i)$ for $\CH(Y_L)\cap \im \sL$ in a similar way.\\

If we write $\tilde{\rho}$ and $\tilde{\sigma}$ for the reduction modulo $2^n$ of $\rho$ and $\sigma$, these bases reduce to bases $(\tilde{e_i})$ of $\CH(X_L,\Z/2^n) \cap \im (\tilde{\rho}_L)_*$ and $(\tilde{f_i})$ of $\CH(Y_L,\Z/2^n) \cap \im (\tilde{\sigma}_L)_*$. In these homogeneous bases the matrix of a correspondence of degree $0$ is diagonal by blocks. The sizes of the blocks are the ranks of the homogeneous components of $\im {\rL}$.

\begin{lemma}
If $r(X,\rho)=1$ then $\disc X $ is trivial.
\end{lemma}
\begin{proof}
Assume $\disc X $ is not trivial. The correspondence $\rho$ induces a projection of $\CH_{d_X}(X_L)$ which is equivariant under the action of $\GalL$. But $\CH_{d_X}(X_L)$ is indecomposable as a $\GalL$-module. It follows that $(\rho_L)_*$ is either the identity or $0$ when restricted to $\CH_{d_X}(X_L)$ , hence $r(X,\rho)\neq 1$.
\end{proof}

\begin{corollary}
\label{triv}
If $r(X,\rho) \neq 2$ then $\GalL$ acts trivially on $\im \rL$.
\end{corollary}

\begin{lemma}
  \label{equal} 
If $r(X,\rho) = 2$ then $r(Y,\rho) = 2$, $\dim Y=\dim X$ and $\disc Y =\disc X $.
\end{lemma}
\begin{proof}
Since the isomorphism $({\alpha_L})_*$ is graded, the $d_X$-th homogeneous component of $\im (\alpha_L)_*$ has rank $2$. This image is a subgroup of the Chow group with coefficients in $\Z /2^n$ of a split quadric, thus the only possibility is that $\dim Y$ is even, $d_X=d_Y$ and $r(Y,\sigma)=2$.

The isomorphism $({\alpha_L})_*$ is equivariant under the action of $\GalL$. It follows that an element of the group $\GalL$ acts trivially on $\CH(X_L,\Z /2^n)$ if and only if it acts trivially on $\CH(Y_L,\Z /2^n)$. But such an element acts trivially on $\CH(X_L,\Z /2^n)$ (\emph{resp.} $\CH(Y_L,\Z /2^n)$) if and only if it acts trivially on the pair of integral cycles $\{l_{d_X},{l_{d_X}}'\} \subset \CH(X_L)$ (\emph{resp.} $\{l_{d_Y},{l_{d_Y}}'\} \subset \CH(Y_L)$). Therefore the pair of integral cycles $\{l_{d_X},{l_{d_X}}'\}$  is $\GalL$-isomorphic to the pair $\{l_{d_Y},{l_{d_Y}}'\}$. By proposition~\ref{disc}, we have a $\GalL$-isomorphism between the split {\'e}tale algebras $\disc X \otimes L$ and $\disc Y \otimes L$. Hence $\disc X$ and $\disc Y$ correspond to the same cocycle in $\H1(\GalL,\Z/2)$, thus are isomorphic.
\end{proof}

\noindent We now proceed with the proof of the injectivity.

Let us first assume that $r(X,\rho) \neq 2$. Then $r(Y,\sigma) \neq 2$ by the preceding lemma. By Corollary \ref{triv} the group $\GalL$ acts trivially on $\im \rL$ and on $\im \sL$, therefore any morphism $(X_L,\rho_L) \to (Y_L,\sigma_L)$ is defined by a cycle invariant under $\GalL$.

Since the isomorphism $\alpha_L$ is of degree $0$, its matrix in our graded bases of the modulo $2^n$ Chow groups is diagonal. Let $\lambda_i \in (\Z /2^n )^{\times}$ be the coefficients in the diagonal so that we have $({\alpha_L})_*(\tilde{e}_i)=\lambda_i \tilde{f}_i$ for all $i$ such that $\CH_i(X_L) \cap \im \rL \neq \emptyset$.

If $r(X,\rho) =1$ then $\lambda_{d_X}$ is defined and we consider the cycle $\beta=(\lambda_{d_X})^{-1} \alpha_L$. If $r(X,\rho)=0$, we just put $\beta=\alpha_L$.

Now we take $k_i \in \Z /2^n$ such that $\lambda_i^{-1}=2 k_i +1$. Let $\Delta \in \End (X_L,\tilde{\rho}_L)$ be the identity morphism. We consider the rational cycle 
\[
\gamma=\Delta + 2 \sum k_i \tilde{e}_i \times \tilde{e}_{\dim X -i},
\] 
where the sum is taken over all $i$ such that $\CH_i(X_L) \subset \im \rL$ (which implies in case $r(X,\rho)=1$ that we do not take $i=d_X$). The composite $\beta \circ \gamma$ is rational, and its matrix in our bases is the identity matrix. This correspondence lifts to an isomorphism with integral coefficients $(X_L,\rho_L) \to (Y_L,\sigma_L)$.\\

Next assume that $r(X,\rho)=2$. Then we have $\dim X=\dim Y$, $r(Y,\sigma)=2$ and $\disc X=\disc Y$ by Lemma~\ref{equal}. The matrix of ${(\alpha_L)}_*$ is diagonal by blocks :
\[
\begin{pmatrix}
\nu_{i_1}&&&&&&\\
&\ddots&&&&&\\
&&\nu_{i_r}&&&&\\
&&&B&&&\\
&&&&\nu_{i_{r+1}}&&\\
&&&&&\ddots&\\
&&&&&&\nu_{i_p}\\
\end{pmatrix}
\]
where $\nu_i \in (\Z /2^n)^{\times}$ and $B \in \GL_2(\Z /2^n)$.\\

Now if $\disc X =\disc Y $ is a field, there is an element in $\GalL$ that simultaneously exchanges the cycles in the bases $\{l_{d_X},{l_{d_X}}'\}$ of $ \CH_{d_X}(X_L,\Z /2^n)$ and $\{l_{d_Y},{l_{d_Y}}'\}$ of $ \CH_{d_Y}(Y_L,\Z/2^n)$. It follows that we may write $B$ as :
\[
\begin{pmatrix} 
a & b\\ 
b & a 
\end{pmatrix}
\]
for some $a$ and $b$ in $\Z /2^n$. The determinant of $B$ is $(a+b)(a-b) \in (\Z /2^n)^{\times}$, hence $(a-b) \in (\Z /2^n)^{\times}$. Thus we may replace $\alpha_L$ by $(a-b)^{-1}\alpha_L$ and assume that $a=b+1$.

As before we may write $\nu_i^{-1}=2 k_i +1$ and replace $\alpha_L$ by the rational cycle :
\[
\alpha_L \circ (\Delta + 2 \sum k_i \tilde{e}_i \times \tilde{e}_{\dim X -i})
\]
the sum being taken over all $i$ such that $\CH_i(X_L) \cap \im \rL \neq \emptyset$ and $i \neq d_X$. Therefore we may assume that $\nu_i=1$ for all $i$ and that we have a matrix of the shape ($\I_r$ being the identity block of size $r$):
\[
\begin{pmatrix} 
\I_s & & &0\\
& \scriptstyle{a+1}&\scriptstyle{a} &\\
& \scriptstyle{a}&\scriptstyle{a+1} &\\
0&&&\I_t\\ 
\end{pmatrix}
\]
The matrix of the rational cycle $h^{d_X}\times h^{d_Y} \in \CH((X\times Y)_L,\Z/2^n)$ is :
\[
\begin{pmatrix} 
0 & & & 0\\
& \scriptstyle{1}&\scriptstyle{1}&\\
& \scriptstyle{1}&\scriptstyle{1} &\\
0&&&0\\ 
\end{pmatrix}
\]
Now $\alpha_L-a (h^{d_X} \times h^{d_Y})$ is rational and its matrix is the identity. This cycle is invariant under $\GalL$ and lifts to an isomorphism $(X_L,\rho_L) \to (Y_L,\sigma_L)$.\\

It remains to treat the case when $\disc X$ is trivial. In this case the group $\GalL$ acts trivially on $\CH(X_L)$ (and on $\CH(Y_L)$ since $\disc Y $ is also trivial). As before, composing with a rational cycle, we may assume that $\nu_i=1$ for all $i$. We write $\det B^{-1}=2 k +1$.

The cycle $\Delta + k (h^{d_X} \times h^{d_X}) \in \End (X_L,\tilde{\rho}_L)$ is rational and its matrix in our basis is :
\[
\begin{pmatrix} 
\I_{p} & & & 0\\
& \scriptstyle{1+k}&\scriptstyle{k}&\\
& \scriptstyle{k}&\scriptstyle{1+k} &\\
0&&&\I_{r}\\ 
\end{pmatrix}
\]
We see that the determinant of this matrix is $1 + 2k$. Therefore the composite $\alpha_L \circ (\Delta + k (h^{d_X}\times h^{d_X}))$  has determinant $1$. We use Lemma \ref{sl} to conclude, which completes the proof of Theorem~\ref{coeff}.

\end{document}